\documentclass[reqno]{article}
\usepackage{tikz}
\usepackage{amsmath,amsthm,amssymb,amsfonts}
\usepackage{mathrsfs}
\usepackage{graphicx}
\usepackage[left=1.5in, right=1.5in, top=1.2in, bottom=1.3in]{geometry}
\usepackage{hyperref}
\usepackage{enumitem}
\setlist{itemsep=-1pt, topsep=1pt}

\makeatletter
\newenvironment{bproof}[1][\proofname]{%
   \par\pushQED{\qed}\normalfont%
   \topsep6\p@\@plus6\p@\relax
   \trivlist\item[\hskip\labelsep\bfseries\itshape#1\@addpunct{.}]%
   \ignorespaces
}{%
   \popQED\endtrivlist\@endpefalse
}
\makeatother

\graphicspath{ {./Steiner/} }
\newtheorem{theorem}{Theorem}[section]
\newtheorem{lemma}[theorem]{Lemma}

\newtheorem{corollary}[theorem]{Corollary}
\newtheorem{conjecture}[theorem]{Conjecture}

\newtheorem{example}[theorem]{Example}
\newtheorem{problem}[theorem]{Problem}

\newtheorem{claim}[theorem]{Claim}

\def\beq{\begin{equation}}\def\eeq{\end{equation}}
\def\beqn{\begin{eqnarray}}\def\eeqn{\end{eqnarray}}

\newcommand{\NN}{\mathbb{N}}

\newcommand{\comment}[1]{}

\newcommand{\SU}{\mathscr{U}}
\newcommand{\red}{\mathrm{red}}
\newcommand{\blue}{\mathrm{blue}}

\begin{document}
\title{Covering 2-colored complete digraphs by monochromatic $d$-dominating digraphs}

\author{Louis DeBiasio\thanks{Department of Mathematics, Miami University, \texttt{debiasld@miamioh.edu}. Research supported in
part by NSF grant DMS-1954170.}, Andr\'as Gy\'arf\'as\thanks{Alfr\'ed R\'enyi Institute of Mathematics, Budapest, P.O. Box 127,
Budapest, Hungary, H-1364. \texttt{gyarfas.andras@renyi.hu}. Research supported in part by NKFIH Grant No. K132696.}}

\maketitle

\begin{abstract}
A digraph is {\em $d$-dominating} if every set of at most $d$ vertices has a common out-neighbor.  For all integers $d\geq 2$, let $f(d)$ be the smallest integer such that the vertices of every 2-edge-colored (finite or infinite) complete digraph (including loops) can be covered by the vertices of at most $f(d)$ monochromatic $d$-dominating subgraphs.  Note that the existence of $f(d)$ is not obvious -- indeed, the question which motivated this paper was simply to determine whether $f(d)$ is bounded, even for $d=2$.  We answer this question affirmatively for all $d\geq 2$, proving $4\leq f(2)\le 8$ and $2d\leq f(d)\le 2d\left(\frac{d^{d}-1}{d-1}\right)\mbox{ for all } d\ge 3$.  We also give an example to show that there is no analogous bound for more than two colors.

Our result provides a positive answer to a question regarding an infinite analogue of the Burr-Erd\H{o}s conjecture on the Ramsey numbers of $d$-degenerate graphs.  Moreover, a special case of our result is related to properties of $d$-paradoxical tournaments.
\end{abstract}

\section{Introduction}

Throughout this note a \emph{directed graph} (or \emph{digraph} for short) is a pair $(V,E)$ where
$V$ can be finite or infinite and $E\subseteq V\times V$ (so in particular, loops are allowed).
A digraph is \emph{complete} if $E=V\times V$.  For $v\in V$, we write $N^+(v)=\{u: (v,u)\in E\}$ and $N^-(v)=\{u: (u,v)\in E\}$.  For a positive integer $k$, we define $[k]:=\{1,\dots,k\}$.
Note that regardless of whether $G=(V,E)$ is a graph or a digraph, if $H=(V', E')$ and $V'\subseteq V$ and $E'\subseteq E$, we will write $H\subseteq G$ and we will always refer to $H$ as a {\em subgraph} of $G$ rather than making a distinction between ``subgraph'' and ``subdigraph.''

Let $G=(V,E)$ be a digraph. For $X,Y\subseteq V$ we say that $X$ \emph{dominates} $Y$ if $(x,y)\in E$ for all $x\in X, y\in Y$.
We say that $G$ is \emph{$d$-dominating} if for all $S\subseteq V$ with $1\leq |S|\leq d$, $S$ dominates some $w\in V$. Note that it is possible for $w\in S$, in which case we must have $(w,w)\in E$.  Reversing all edges of a \emph{$d$-dominating} digraph gives a \emph{$d$-dominated} digraph. These notions are well studied for tournaments (see Section \ref{kpara}).

A {\em cover} of a digraph $G=(V,E)$ is a set of subgraphs $\{H_1, \dots, H_t\}$ such that $V(G)=\bigcup_{i\in [t]} V(H_i)$.  By a {\em $2$-coloring} of $G=(V,E)$, we will always mean a $2$-coloring of the edges of $G$; i.e. a function $c:E\to [2]$.
Given a $2$-coloring of $G$, we let $E_i$ be the set of edges receiving color $i$ (i.e. $E_i=c^{-1}(\{i\})$) and $G_i=(V,E_i)$ for $i\in [2]$.  A {\em cover of $G$ by monochromatic subgraphs} is a cover $\{H_1, \dots, H_t\}$ of $G$ such that for all $i\in [t]$ there exists $j\in [2]$ such that $H_i\subseteq G_j$.

The following problem was raised in \cite[Problem 6.6]{CDM}.

\begin{problem}\label{2dom} Given a $2$-colored complete digraph $K$, is it possible to cover $K$ with at
most four monochromatic $2$-dominating subgraphs? (If not four, some
other fixed number?)
\end{problem}

Our main result is a positive answer for the qualitative part of Problem \ref{2dom} in a more general form.

\begin{theorem}\label{dbounded} Let $d$ be an integer with $d\geq 2$.  In every $2$-colored complete
digraph $K$, there exists a cover of $K$ with at most $2\times \sum_{i=1}^d d^i=2d\left(\frac{d^{d}-1}{d-1}\right)$ monochromatic $d$-dominating subgraphs. In case of
$d=2$ there exists a cover of $K$ with at most eight monochromatic $2$-dominating subgraphs.
\end{theorem}

For all integers $d\geq 1$, let $f(d)$ be the minimum number of
monochromatic $d$-dominating subgraphs needed to cover an arbitrarily 2-colored complete digraph.  Note that obviously $f(1)=2$ since the two sets of monochromatic loops provide an optimal cover. For $d\ge 2$, Theorem \ref{dbounded} shows that $f(d)$ is well-defined. Example \ref{exCIM} below (adapted from \cite[Proposition 6.3]{CDM}) combined with Theorem \ref{dbounded} gives
\begin{equation}\label{fbound}
4\leq f(2)\leq 8 ~\text{ and }~ 2d\leq f(d)\leq 2d\left(\frac{d^{d}-1}{d-1}\right) \text{ for all integers } d\geq 3.
\end{equation}

\begin{example}\label{exCIM}
Let $K$ be a complete digraph on at least $2d$ vertices and partition $V(K)$ into non-empty sets $R_1, \dots, R_d$
 and $B_1, \dots, B_d$, color all edges inside $R_i$ red, all edges inside $B_i$ blue, all edges from $R_i$ to $B_j$
 red, all edges from $B_i$ to $R_j$ blue, all edges between $R_i$ and $R_j$ with $i\neq j$ blue, and all edges between
 $B_i$ and $B_j$ with $i\neq j$ red.  One can check that every monochromatic $d$-dominating subgraph of $K$ is entirely
 contained inside one of the sets $R_1, \dots, R_d, B_1, \dots, B_d$.
\end{example}

Finally, the following example shows that for $d\ge 2$ there is no analogue of Theorem \ref{dbounded} for more than two colors  (c.f. \cite[Example 2.3]{CDM}).

\begin{example} Let $V$
be a totally ordered set and let $K$ be the complete digraph on $V$ where for all $i\in V$, $(i,i)$ is green and for all
$i,j\in V$ with $i<j$, $(i,j)$ is red and $(j,i)$ is blue. Note that for $d\ge 2$ the only monochromatic $d$-dominating subgraphs are the green loops and thus no bound can be put on the number of monochromatic $d$-dominating subgraphs needed to cover $V$.
\end{example}

\subsection{Motivation}

A graph $G$ is \emph{$d$-degenerate} if there is an ordering of the vertices $v_1, v_2, \dots$ such that for all
 $i\geq 1$, $|N(v_i)\cap \{v_1,\dots, v_{i-1}\}|\leq d$ (equivalently, every subgraph has a vertex of degree at
 most $d$). Burr and Erd\H{o}s conjectured \cite{BE} that for all positive integers $d$, there exists $c_d>0$ such
 that every 2-coloring of $K_n$ contains a monochromatic copy of every $d$-degenerate graph on at most $c_dn$ vertices.
  This conjecture was recently confirmed by Lee \cite{Lee}.

The motivation for Problem \ref{2dom} relates to the following conjecture also raised in \cite[Problem 1.5, Conjecture 10.2]{CDM}
which can be thought of as an infinite analogue of the Burr-Erd\H{o}s conjecture.

\begin{conjecture}\label{degen}
For all positive integers $d$, there exists a real number $c_d>0$ such that if $G$ is a countably infinite $d$-degenerate graph
with no finite dominating set, then in every $2$-coloring of the edges of $K_{\mathbb{N}}$, there exists a monochromatic copy of
$G$ with vertex set $V\subseteq \mathbb{N}$ such that the upper density of $V$ is at least $c_d$.
\end{conjecture}

The case $d=1$ was solved completely in \cite{CDM} (regardless of whether $G$ has a finite dominating set or not). For certain 2-colorings of $K_{\NN}$, described below, Theorem \ref{dbounded} implies a positive solution to Conjecture \ref{degen} for $d\geq 2$.

Suppose that for some finite subset $F\subseteq \NN$, we have a partition of $\NN\setminus F$ into (finitely or infinitely
many) infinite sets $\mathcal{X}=\{X_1, \dots, X_n,\dots\}$.  Also suppose that we have ultrafilters
$\SU_1, \SU_2, \dots, \SU_n, \dots$ on $\NN$ such that for all $i\geq 1$, $X_i\in \SU_i$.  Finally, suppose that for all
$i, j\geq 1$ there exists $c_{i,j}\in [2]$ such that for all $v\in X_i$, $\{u: \{u,v\} \text{ has color } c_{i,j}\}\cap X_j\in \SU_j$.  This last condition ensures
that if there exists $X_{i_1}, \dots, X_{i_n}$ and $X_j$ such that $c_{i_1, j}=\dots=c_{i_n, j}=:c$, then every finite
collection of vertices in $X_{i_1}\cup \dots \cup X_{i_n}$ has infinitely many common neighbors of color $c$ in $X_j$.  Note that such a scenario can be realized as follows: For all $i,j$, let $c_{i,j}\in [2]$
and color the edges from $X_i$ to $X_j$ so that every vertex in $X_i$ is incident with cofinitely many edges of color
 $c_{i,j}$ (by using the half graph coloring\footnote{Given a totally ordered set $Z$ and disjoint $X,Y\subseteq Z$ the \emph{half graph coloring} of the complete bipartite graph
$K_{X,Y}$ is a 2-coloring of the edges of $K_{X,Y}$ where for all $i\in X$, $j\in Y$, $\{i,j\}$ is red if and only if  $i\leq j$.} when $c_{i,j}\neq c_{j,i}$ for instance).

The above coloring of $K_{\NN}$ naturally corresponds to a 2-colored complete digraph in the following way:
Let $K$ be a 2-colored complete digraph on $\mathcal{X}$ where we color $(X_i,X_j)$ with color $c$ if for all $v\in X_i$,
$\{u: \{u,v\} \text{ has color } c\}\cap X_j\in \SU_j$.  Now by Theorem \ref{dbounded}, $K$ can be covered by $t\leq f(d+1)$ monochromatic $(d+1)$-dominating subgraphs $G_1, \dots, G_t$.  Since $\NN\setminus F=\bigcup_{i\in [t]}\left(\bigcup_{X\in V(G_i)}X\right)$, there exists $i\in [t]$ such that $V_i:=\bigcup_{X\in V(G_i)}X$ has upper density at least $1/f(d+1)$.  Without loss of generality, suppose the edges of $G_i$ are red.  By the construction, $V_i$ has the property that
for all $S\subseteq V_i$ with $1\leq |S|\leq d+1$, there is an infinite subset $W\subseteq V_i$ such that every
edge in $E(S, W)$ is red.  As shown in \cite[Proposition 6.1]{CDM}, if $G$ is a graph satisfying the hypotheses of Conjecture \ref{degen}, then
there exists a red copy of $G$ which spans $V_i$ and thus has upper density at least $1/f(d+1)$.

\section{Covering digraphs, proof of Theorem \ref{dbounded}}

For a graph $G$, we denote the order of a largest clique (pairwise adjacent vertices) in $G$ by $\omega(G)$.
Given a 2-colored complete digraph $K$ and a set $U\subseteq V(K)$, define $G[U]_{\blue}$ to be the graph on $U$
where $\{u,v\}\in G[U]_{\blue}$ if and only if $(u,v)$ and $(v,u)$ are blue in $K$; define $G[U]_{\red}$ analogously.

Given positive integers $\omega$ and $d$, let $f(\omega, d)$ be the smallest positive integer $D$ such that if $K$
is a 2-colored complete digraph on vertex set $V$ where every loop has the same color, say red, and $\omega(G[V]_{\blue})= \omega$,
then $V$ can be covered by at most $D$ monochromatic $d$-dominating subgraphs.  Also define $f(0, d)=0$.

\begin{lemma}\label{L1}~
\begin{enumerate}
\item $f(1,2)=1$
\item $f(\omega, d)\leq d(f(\omega-1, d)+1)$ for all $1\leq \omega \leq d$ (in particular, $f(1,d)\leq d$).
In fact, all $d$-dominating subgraphs in the covering have the same color as the loops.
\end{enumerate}
\end{lemma}

Note that the upper bound $\omega\leq d$ is not strictly necessary, but we include it here for clarity since in the next lemma, we will prove a stronger result when $\omega\geq d+1$.

\begin{proof}
Let $K$ be a 2-colored complete digraph on vertex set $V$ where all loops have the same color, say red.

(1) is trivial since for all distinct $u,v\in V$ both $(u,u)$ and $(v,v)$ are red and $\omega(G[V]_{\blue})=1$ implies that either $(u,v)$
or $(v,u)$ is red.

To see (2), note first that we may assume that $K$ itself is not spanned by a red $d$-dominating subgraph, otherwise we are done.
This is witnessed by a set $U=\{u_1,\dots,u_d\}
\subseteq V$, such that there is no $w\in V$ with $(u_i,w)$ red for
all $i\in [d]$.

For all $i\in [d]$ we define
$$W_i=\{v\in V: (v,u_i) \mbox{ is red}\}.$$
Note that $u_i\in W_i$ and $K[W_i]$ is spanned by a red $d$-dominating subgraph for all
$i\in [d]$.

Set $V'=V\setminus (\cup_{i\in [d]} W_i)$ and define
$$T_i=\{v\in V': (u_i,v) \mbox{ is blue}\}.$$
Note, that by the definition of $V'$, $(v,u_i)$ is also blue for
all $v\in T_i$ and $i\in [d]$.  Moreover, from the selection of
$U$, every vertex in $V'$ receives a blue edge from some vertex in
$U$ and therefore $V'=\cup_{i=1}^d T_i$.

Note that if $\omega=1$, then $T_i=\emptyset$ for all $i\in [d]$ and
thus $\cup_{i\in [d]} W_i$ is a cover of $K$ with $d$ red
$d$-dominating subgraphs; i.e. $f(1, d)\leq d=d(f(0,d)+1)$.

Otherwise, we have that $\omega(K[T_i]_{\blue})\leq \omega-1$ and thus $K$
is covered by at most $$d+d\cdot f(\omega-1,d)=d(f(\omega-1,d)+1)$$
red $d$-dominating subgraphs.
\end{proof}

\begin{lemma}\label{L2}
Let $K$ be a 2-colored complete digraph $K$ where $R$ is the set of red loops and $B$ is the set of blue loops.
If $\omega(G[R]_{\blue})\geq d+1$, then $V(K)$ can be covered by at most $d$ red $d$-dominating subgraphs and at most
one blue $d$-dominating subgraph.
Likewise, if $\omega(G[B]_{\red})\geq d+1$.  In particular, this implies $f(\omega, d)\leq d+1$ for $\omega\geq d+1$.
\end{lemma}

\begin{proof}
Suppose $\omega(G[R]_{\blue})\geq d+1$ and let $X=\{x_1, \dots, x_d, x_{d+1}\}\subseteq R$ be a set of order $d+1$ which witnesses
this fact. For
$i\in [d]$ we define
$$W_i=\{v\in V(K): (v,x_i) \mbox{ is red}\}.$$
Note that $x_i\in W_i$ and $K[W_i]$ is spanned by a red $d$-dominating subgraph
for all $i\in [d]$.

Set $V'= X\cup (V(K)\setminus (\cup_{i\in [d]} W_i))$ and note that
for all $v\in V'$, $[v, X]$ is blue.  Now let $S\subseteq V'$ such
that $1\leq |S|\leq d$.  If $S\subseteq X$, then since $|S|<|X|$,
there exists $x_i\in X\setminus S$ such that every edge in $[S,
x_i]$ is blue; otherwise $|S\cap X|\leq d-1$ and there exists $i\in
[d]$ such that $x_i\notin S$ and every edge in $[S, x_i]$ is blue.
So there is one blue $d$-dominating subgraph which covers $V'$, which together with the red $d$-dominating subgraphs
$K[W_1], \dots, K[W_d]$ gives the result.

When $\omega(G[B]_{\red})\geq d+1$, the proof is the same by switching the colors.
\end{proof}

Now we are ready to prove our main result.

\begin{bproof}[Proof of Theorem \ref{dbounded}] Let $V(K)=R\cup B$ where $R,B$ are the vertex sets of the red and blue loops, respectively.
If $\omega(G[R]_{\blue})\geq d+1$ or $\omega(G[B]_{\red})\geq d+1$, then by Lemma \ref{L2}, $R\cup B$ can be covered by at most
$d+1$ monochromatic
$d$-dominating subgraphs.  So suppose
$\omega(G[R]_{\blue})\leq d$ and $\omega(G[B]_{\red})\leq d$.  Now by Lemma \ref{L1}, each of $K[R]$ and $K[B]$ can be covered by at most
$4$ monochromatic $d$-dominating subgraphs when $d=2$, and by at most $\sum_{i=1}^\omega d^i\leq \sum_{i=1}^d d^i$ monochromatic
$d$-dominating subgraphs when $d\geq 3$.
\end{bproof}

\section{Paradoxical tournaments}\label{kpara}

In the above section, we proved that $f(1,2)=1$ and $f(1,d)\leq d$ for all $d\geq 3$.  Naturally, we wondered if the
upper bound on $f(1,d)$ could be improved when $d\geq 3$ (since any improvement on $f(1,d)$ would improve the general upper bound on $f(d)$).  In this section we show that it cannot; that is, $f(1,d)=d$ for all $d\geq 3$.

A {\em tournament} is a digraph $(V,E)$ such that for all distinct
$x,y\in V$ exactly one of $(x,y),(y,x)$ is in $E$ and $(x,x)\notin
E$.  Given a digraph $G=(V,E)$, we say that $S\subseteq V$ is an out-dominating set if for all $v\in V\setminus S$, there exists $u\in S$ such that $(u,v)\in E$, and we say that $S$ is an in-dominating set if for all $v\in V\setminus S$, there exists $u\in S$ such that $(v,u)\in E$.  Note that a tournament $T$ is $d$-dominating ($d$-dominated) if and only if $T$ has no in-dominating (out-dominating) set of order $d$.

We call a $d$-dominating ($d$-dominated) tournament {\em critical} if its proper
subtournaments are not $d$-dominating ($d$-dominated). For a tournament $T$, let $T^*$ be the digraph obtained from $T$ by adding a loop at every vertex.

Our main result of this section is the following.

\begin{theorem}\label{critpara} For all integers $d\geq 2$, if $T$ is a critical $d$-dominated tournament with no $(d+1)$-dominating subtournaments, then $f(1,d+1)=d+1$.
\end{theorem}

However, before proving Theorem \ref{critpara}, we show that such a tournament exists for all $d\geq 2$ from which we obtain the following corollary.

\begin{corollary}
For all $d\geq 3$, $f(1,d)=d$.
\end{corollary}

Note that the absence of loops and two-way oriented edges make
the existence of $d$-dominated tournaments a nontrivial problem. This existence problem for $d$-dominated
tournaments was proposed by Sch\"utte (see \cite{E}) and was first
proved by Erd\H os \cite{E} with the probabilistic method, then
Graham and Spencer \cite{GS} gave an explicit construction using sufficiently large Paley
tournaments\footnote{For a prime power $p$, $p\equiv -1 \pmod 4$, the Paley
tournament $QT_p$ is defined on vertex set $V=[0,p-1]$ and $(a,b)$
is a directed edge if and only if $a-b$ is a non-zero square in the finite field $\mathbb{F}_p$.}.

Note that Babai
\cite{B} coined the term {\em $d$-paradoxical tournament} for what we refer to as $d$-dominated tournament.  In this spirit, we say that a tournament is {\em perfectly $d$-paradoxical} if it is $d$-dominating, $d$-dominated, has no $(d+1)$-dominating subtournaments, and has no $(d+1)$-dominated subtournaments.  A result of Esther and George Szekeres \cite{KSZ} combined with the fact that Paley tournaments are self-complementary implies that $QT_7$ is perfectly $2$-paradoxical and $QT_{19}$ is perfectly $3$-paradoxical. It is an open question (which to the best of our knowledge we are raising here for the first time) whether every Paley tournament is perfectly $d$-paradoxical for some $d$. While we can't settle that question, the following beautiful example of Bukh \cite{Bu} shows that perfectly $d$-paradoxical tournaments exist for all $d\geq 2$.  We repeat his proof here (tailored to the terminology of this paper) for completeness.

\begin{example}[Bukh \cite{Bu}]\label{pparadox}
For all integers $d\geq 2$, there exists a perfectly $d$-paradoxical tournament.  In particular, there exists a critical $d$-dominated tournament which has no $(d+1)$-dominating subtournaments.
\end{example}

\begin{proof}
Let $d$ be an integer with $d\geq 2$ and let $n=m(d+1)$ where $m=2^{3d}$.  Let $V=\{0, 1, \dots, n-1\}$ and let $G$ be the oriented graph on $V$ where $(i,j)\in E(G)$ if and only if $1\leq j-i\leq m-1$ (with addition modulo $n$).  In other words $G$ is the oriented $(m-1)$st power of a cycle on $n$ vertices.  Now we define a tournament $T$ by starting with the oriented graph $G$ and for all distinct $i,j\in V$, if $(i,j), (j,i)\not\in E(G)$, then independently and uniformly at random let $(i,j)\in E(T)$ or $(j,i)\in E(T)$.

First note that every induced subgraph of $G$ has an in-dominating set of order at most $d+1$ and an out-dominating set of order at most $d+1$ and thus the same is true of every subtournament of $T$.  This implies that $T$ has no $(d+1)$-dominating subtournaments and no $(d+1)$-dominated subtournaments.

Now we claim that with positive probability, $T$ has no out-dominating sets of order $d$ and no in-dominating sets of order $d$ and thus $T$ is $d$-dominated and $d$-dominating.  Let $S\subseteq V$ with $|S|=d$ and let $$N^+_G[S]=\{v\in V: v\in S \text{ or there exists } u\in S \text{ such that } (u,v)\in E(G)\}.$$  Let $V':=V\setminus N^+_G[S]$ and note that $|N^+_G[S]|\leq dm$ and thus $|V'|\geq m$.  The probability that $v\in V'$ is dominated by $S$ in $T$ is $1-2^{-d}$ and thus the probability that every vertex in $V'$ is dominated by $S$ is $(1-2^{-d})^{|V'|}\leq (1-2^{-d})^{m}\leq e^{-2^{-d}m}=e^{-4^d}$.  Likewise for every vertex of $V'$ dominating $S$.  So the expected number of out-dominating or in-dominating sets of order $d$ is at most
$$2\binom{n}{d}e^{-4^d}<2(em)^de^{-4d}<2(e^{3d+1})^de^{-4^d}<1$$
(where the last inequality holds since $(3d+1)d<4^d$ for all $d\geq 2$), which establishes the claim.

Starting with a perfectly $d$-paradoxical tournament $T$, let  $T'$ be a minimal subtournament of $T$ which is $d$-dominated.  So $T'$ is critical $d$-dominated and has no $(d+1)$-dominating subtournaments.
\end{proof}

The proof of Theorem \ref{critpara} will follow from two more general lemmas.

\begin{lemma}\label{critpara2} Let $T$ be a tournament and let $d\geq 2$.  If $T$ is 2-dominating and there exists a set $W\subseteq V(T)$ with $|W|=d$ such that $W$ dominates exactly one vertex $v$, then $T^*$ is not $(d+1)$-dominating.
In particular, if $T$ is critical $d$-dominating, then $T^*$ is not $(d+1)$-dominating.
\end{lemma}

\begin{proof}
Let $W=\{w_1, \dots, w_d\}$ and $v$ be as in the statement.  To see that $T^*$ is not $(d+1)$-dominating, it is enough to prove that for some $u\in N^+(v)$ the set $W\cup\{u\}$ does not dominate any vertex in $T^*$ (note that since $T$ is 2-dominating, $N^+(v)\neq \emptyset$).  Suppose for contradiction that this is not the case; that is, for all $u\in N^+(v)$ the set $W\cup\{u\}$ dominates some vertex $x$ in $T^*$.  Note that by the definition of $W$ and the fact that $u\in N^+(v)$, it must be the case that $x\in W$; without loss of generality, suppose $x=w_1$.  This implies that for all $i\in [d]$, $(w_i,w_1)\in E(T)$.  But now this implies that for all $u\in N^+(v)$, $W\cup \{u\}$ dominates $w_1$.  On the other hand since $T$ is 2-dominating, it must be the case that there exists a vertex which is dominated by $\{w_1, v\}$ in $T$, but every outneighbor of $v$ is an inneighbor of $w_1$ and thus we have a contradiction.

To get the second part of the lemma, first note that if $T$ is critical $d$-dominating, then $T$ is 2-dominating.  Moreover, for all $v\in V$, since $T-v$ is not $d$-dominating there exists $W=\{w_1,\dots w_d\}\subseteq V(T)\setminus \{v\}$ which does not dominate any vertex in $V(T)\setminus \{v\}$, but since $T$ is $d$-dominating, $W$ must dominate $v$.
\end{proof}

If $G=(V,E)$ is a digraph
such that there exists $w\in V$ such
that $(v,w)\in E$ for all $v\in V$ (including $v=w$), then note that $G$ is $d$-dominating for all $d\le |V|$.  In this case we call $G$ {\em trivially $d$-dominating}.

\begin{lemma}\label{tlemma} Let $T$ be a tournament.  If $T$ is critical $d$-dominating,
then $T^*$ cannot be covered by less than $d+1$
$(d+1)$-dominating subgraphs.
\end{lemma}

\begin{proof} Suppose for contradiction that for some $t\leq d$ there are $(d+1)$-dominating
subgraphs $H_1,\dots,H_t$ which cover $T^*$.  Since $T$ is critical $d$-dominating we have by Lemma \ref{critpara2} that $T^*$ is not $(d+1)$-dominating, and thus all
$V(H_i)$ are proper subsets of $V(T^*)$.

\begin{claim}\label{cl_triv} Each $H_i$ is trivially $(d+1)$-dominating.
\end{claim}
\begin{proof} The claim is obvious if
$|V(H_i)|\leq d$; so suppose that $|V(H_i)|\ge d+1$. Since $T$ is critical $d$-dominating, the
subtournament $T_i$ of $T$ spanned by $V(H_i)$ is not
$d$-dominating. This is witnessed by a set $W=\{w_1,\dots,w_d\}\subseteq
V(T_i)$  such that $W$ does not dominate any vertex in
$U=V(T_i)\setminus W$.  Let $u\in U$.  Since $H_i$ is
$(d+1)$-dominating, $W\cup u$ dominates some vertex $x\in V(H_i)$
which must be in $W$ from the definition of $W$.  Without loss of generality, let
$x=w_1$. This implies that $(u,x_1)\in E(T)$ and for all $i\in [d]$, $(w_i,w_1)\in E(T)$.  But now this implies that for all $u\in U$, $W\cup \{u\}$ dominates $w_1$ and thus all vertices of
$V(H_i)$ (including $w_1$) are oriented to $w_1$ proving the claim.
\end{proof}

Claim \ref{cl_triv} implies that for all $i\in [t]$ there is a
vertex $v_i\in V(H_i)$ which is dominated by all vertices of $H_i$.
But since $\cup_{i=1}^t V(H_i)=V(T)$, the set $\{v_1,\dots,v_t\}$
does not dominate any vertex in $T$, contradicting the fact that $T$
is $d$-dominating.
\end{proof}

\begin{bproof}[Proof of Theorem \ref{critpara}]
First note that $f(1,d+1)\le d+1$ by Lemma \ref{L1}.

Let $T_B$ be a tournament on vertex set $V$ such that $T_B$ is critical $d$-dominated and has no $(d+1)$-dominating subtournaments.  Define the $2$-colored complete digraph $K$ on $V$ by coloring all edges of $T_B$ blue, and all edges of $(V\times V)\setminus E(T_B)$ red.  Let $T_R$ be the tournament with $E(T_R)=\{(y,x):(x,y)\in E(T_B)\}$ and note that every edge of $T_R$ is red and $T_R$ has no loops.  Since $T_B$ is critical $d$-dominated, this implies that $T_R$ is critical $d$-dominating (since $T_R$ is obtained by reversing all the edges of $T_B$).

Note that by the assumption on $T_B$, every monochromatic $(d+1)$-dominating subgraph in $K$ must be red.  However, since $T_R$ is crtical $d$-dominating, we get that $f(1,d+1)\ge d+1$ from Lemma \ref{tlemma}.
\end{bproof}

\paragraph{Acknowledgements.}

We thank Boris Bukh for Example \ref{pparadox} and for his comments on the paper.

\end{document}